\documentclass[reqno,11pt,a4paper]{amsart}
\usepackage{amsmath,amssymb,latexsym}

\usepackage{enumerate}
\usepackage[all]{xy}
\usepackage{fancyhdr}
\usepackage{color}

\usepackage{graphics}
\usepackage{hyperref}

\newcommand{\Z}{\mbox{$\mathbb Z$}}
\newcommand{\Q}{\mbox{$\mathbb Q$}}

\newcommand{\lrta}{\longrightarrow}

\newcommand{\g}{\mathcal{G}}

\newcommand{\K}{K_\infty}

\newcommand{\QQ}{\mathbb{Q}}
\newcommand{\N}{\mathbb{N}}
\newcommand{\ZZ}{\mathbb{Z}}

\newcommand{\Hom}{\mathrm{Hom}}

\newcommand{\gal}{\mathrm{Gal}}
\newcommand{\cyc}{\mathrm{cyc}}

\newcommand{\lra}{\longrightarrow}

\newcommand{\plim}{\varprojlim}

\usepackage[OT2,OT1]{fontenc}
\newcommand\cyr{%
\renewcommand\rmdefault{wncyr}
\renewcommand\sfdefault{wncyss}
\renewcommand\encodingdefault{OT2}
\normalfont\selectfont}

\DeclareTextFontCommand{\textcyr}{\cyr}

\newtheorem{theorem}{Theorem}[section]
\newtheorem{proposition}[theorem]{Proposition}
\newtheorem{lemma}[theorem]{Lemma}
\newtheorem{remark}[theorem]{Remark}
\newtheorem{corollary}{Corollary}
\newtheorem{definition}[theorem]{Definition}

\newtheorem{conjecture}[theorem]{Conjecture}

\newtheorem*{theorem*}{Theorem}

\newtheorem*{maintheorem}{Main Theorem}

\hypersetup{
   colorlinks=true, 
    linktoc=all,     
    linkcolor=blue,
    citecolor=blue,  }

\newcommand{\Keywords}[1]{\par\noindent
{\small{Keywords and phrases}: #1}}

\newcommand{\AMS}[1]{\par\noindent
{\small{AMS Subject Classification}: #1}}

\author{Somnath Jha, Sudhanshu Shekhar}
\address{Department of Mathematics and Statistics, IIT Kanpur, Kanpur 208016, India}
\email{jhasom@iitk.ac.in, sudhansh@iitk.ac.in }

\begin{document}

\title{Non-commutative twisted Euler characteristic}

\begin{abstract}
It is well known that given a finitely generated torsion module $M$ over the Iwasawa algebra $\Z_p[[\Gamma ]]$, where $\Gamma \cong \Z_p$, there exists a continuous $p$-adic character $\rho$ of $\Gamma$ such that, for the twist $M(\rho)$ of $M$, the $\Gamma_n : = \Gamma^{p^n}$ Euler characteristic i.e. $\chi(\Gamma_n,  M(\rho))$ is finite for every $n$. We prove a generalization of this result by considering modules over the Iwasawa algebra of a general $p$-adic Lie group $G$, instead of $\Gamma$.  We relate this twisted Euler characteristic to the evaluation of the {\it Akashi Series} at the twist and in turn use it to indicate some application to the Iwasawa theory of elliptic curves. This article is a natural generalization of the result established in \cite{joz}. 
\end{abstract}

\maketitle
\tableofcontents
\let\thefootnote\relax\footnotetext{

\AMS{11R23, 14H52}
}

\section{Introduction}\label{intro}
In this article, we study certain topics in non-commutative Iwasawa theory following the set up of \cite{CFKSV}.  Fix an odd integer prime $p$. Let $L$ be a finite extension of  $\QQ_p$ and let $O$ denote the ring of integers of $L$. Let $G$ be a compact $p$-adic  Lie group without an element of order $p$ and $H$ be a closed subgroup of $G$ such that $\Gamma := G/H \cong  \ZZ_p$. We fix a topological generator $\gamma$ of $\Gamma$ and a lift $\tilde{\gamma}$ of $\gamma$ in $G$.  For a general profinite group $\g$, let $\Lambda_{O}(\g)$ denote the Iwasawa algebra defined as 
\[  \Lambda_{O}(\g) := \plim _U O[\g/U] \]
where $U$ varies over open normal subgroups of $\g$ and inverse limit is taken with respect to natural projection maps.  Given any ring $R$ and a left $R$ module $M$, we denote by $M[p^r]$ the set of $p^r$-torsion elemets of $M$ and we define $M(p) := \underset{r \geq 1}{\cup}M[p^r]$. Let $\mathfrak M_H(G)$ denote the category of finitely generated $\Lambda_{O}(G)$-modules $M$ such that $M/M(p)$ is a finitely generated $\Lambda_{O}(H)$-module. For a $\Lambda_{O}(G)$-module $M$ and a continuous character $\rho : \Gamma \lrta \ZZ_p^\times $, we denote by $M(\rho)$ the $\Lambda_{O}(G)$-module $M\otimes_{\ZZ_p} \ZZ_p(\rho)$ with diagonal $G$-action. 
Recall for a compact $p$-adic Lie group  $G$  which has no element of order $p$,  the Iwasawa algebra $\Lambda_{O}(G)$ has finite global dimension (\cite{B}, \cite{la}).

\begin{definition}
Let $G$ be a compact $p$-adic Lie group without any element of order $p$. For a finitely generated $\Lambda_{O}(G)$-module $M$, we say the $G$-Euler characteristic of $M$ exists if the homology groups $H_i(G,M)$ are finite for every $i \geq 0$. When the $G$-Euler characteristic of $M$ exists, then it is defined as 
\[  \chi(G,M) :=  \prod_i (\#H_i(G,M))^{(-1)^i} \]
\end{definition}

Given a $\Lambda_{O}(G)$-module $M$, $\chi(G,M)$ is an  invariant attached to $M$ and for an appropriate motive $\mathcal M$ and  $G$, under suitable hypotheses, the $G$-Euler characteristic of  certain Selmer group  $S(\mathcal M)$ attached to $\mathcal M$ i.e. $\chi(G, S(\mathcal M))$ often encodes important arithmetical properties of $\mathcal M$. We discuss a concrete example. Let $E/\Q$ be an elliptic curve with  good, ordinary reduction at $p \geq 11$ and $E$ has no prime of additive reduction.
Let $\Q_\cyc$ denotes the cyclotomic $\Z_p$ extension  of $\Q$ and $S_p(E/\Q)$ denote the $p^\infty$-Selmer group  of $E$ over $\Q$ (defined in \S   \ref{apeq}). 
We assume $L_E(s)$, the Hasse-Weil (complex) $L$-function of $E$ over $\Q$  satisfies $L_E(1) \neq 0$. Then Birch and Swinnerton-Dyer conjecture describing the leading term of $L_E(s)$ at $1$ together with \cite[Theorem 4.1]{Gr} implies that 
\begin{equation}\label{sel} 
  \chi(\Gamma,S_p(E/\Q_{\cyc})^\vee) =_p  \frac{L_E(1)}{\Omega_E}  (\# \tilde{E}(\mathbb F_p)(p))^2, 
\end{equation}
Here, $=_p$ denotes the equality up to a $p$-adic unit and $\Omega_E$ denote the smallest positive real period of $E$.  Note that due to works of  Gross-Zagier, Kolyvagin and others, it is established that  $L_E(1) \neq 0$ implies $E(\Q)$ is finite. Then for such an  $E$ and $p $ as above, the leading term formula for $L_E(s)$ in BSD  is derived from the proof of Iwasawa main conjecture for $E$ at $p$ due to works of Kato, Skinner-Urban and others (see \cite[Theorem 3.35(a)]{su}) for details). As a consequence of this, the formula in \eqref{sel} is established using \cite[Conjecture 1.13 \& \S 4]{Gr}.  

On the other hand, for a general $p$-adic Lie extension $K_\infty$ of a number field $K$ with Galois group $G$ as above and for an elliptic curve $E/\Q$, the existence of $p$-adic $L$-function of $E$ over $K_\infty$ is not known and even the algebraicity  of the $L$-values has not been established in general.  However, a  formulation of Iwasawa main conjecture is given in \cite{CFKSV} for  dual Selmer groups  in $\mathfrak  M_H(G)$ and in this context twisted Euler characteristic of modules in $\mathfrak  M_H(G)$ via Akashi series has been extensively studied in \cite{CFKSV}. In particular, for a general $\rho: G \lra \mathrm{GL}_n(\Z_p)$, the conjectural relation between $\chi(G, S_p(E/K_\infty)^\vee(\rho))$ and special values of twisted $L$-function is discussed in detail  (cf. \cite[Theorem 3.6]{CFKSV}). The main result of this article is the following theorem.

\begin{maintheorem}\label{main}
Let $M$ be a $\Lambda_{O}(G)$-module which is in $\mathfrak M_H(G)$ i.e. $M/M(p)$ is finitely generated over $\Lambda_{O}(H)$. Then there exists a continuous character $\rho : \Gamma \rightarrow \ZZ_p^\times $  such that $\chi(U,M(\rho))$ exists for every open normal subgroup $U$ of $G$.   \qed 
\end{maintheorem}
\begin{remark}
Actually, we will establish a slightly stronger result; we will prove there is a countable subset $S_M$ of the set $S : = \{ \text{all continuous characters from  }\Gamma \lra \ZZ_p^\times\}$ such that if we choose and fix any $\rho \in S \setminus S_M$, then $\chi(U,M(\rho))$ exists for every open normal subgroup $U$ of $G$.
\end{remark}

Our Main Theorem  is a generalization and uses the proof of the following  result proved in \cite{joz}, coauthored by the first named author. 
\begin{theorem}[\cite{joz}]\label{known}
Let $M$ be a $\Lambda_{O}(G)$-module which is finitely generated over $\Lambda_{O}(H)$. Then there exists a continuous character $\rho : \Gamma \lrta \ZZ_p^\times $, such that    $  M(\rho)_U:= H_0(U, M)$ is finite for every open normal subgroup $U$ of $G$. \qed

\end{theorem}

\begin{remark} (1) For $G = \Gamma$, $H =1$, Theorem \ref{known} and Main Theorem are equivalent. Indeed it is the well known ``twisting lemma" which can be found in works of \cite{gr1}, \cite{pr}.

(2) More generally,   if $G$ is a `finite by nilpotent' group, then Theorem \ref{known} and Main Theorem  are equivalent (cf. \cite{wa}). Examples of `finite by nilpotent' group includes  $G \cong \ZZ_p^{d}$, $d \in \N$.

(3) However for a  general compact $p$-adic Lie group $G$ without any element of order $p$ and for a finitely generated $\Lambda_{O}(G)$ module $M$, $M_G$ is finite does not necessarily  imply  $\chi(G, M)$ exists (cf. \cite{bh}). Thus indeed Main Theorem  is more general than Theorem \ref{known}.

(4) The assumption $G$ has no element of order $p$ is assumed in Main Theorem but it is not needed in  Theorem \ref{known} of \cite{joz}. However in our case this assumption is clearly necessary; otherwise $\Lambda_{O}(G)$ may not have finite global dimension and hence the Euler characteristic may not be defined.

(5) It is necessary to assume $M \in \mathfrak M_H(G)$ in Main Theorem. See \cite[example, page 194]{CFKSV}, where for certain $M $ and $G$ with $M$ not in $\mathfrak M_H(G)$, even $\chi(G, M(\rho))$ does not exist for any $\rho$.

\end{remark} 

In  Section \ref{2section}, we prove the Main Theorem. In subsection \ref{appli1}, we discuss the relation between the Main Theorem and Akashi series and apply it to deduce Corollary \ref{eva1}. We then apply Main Theorem and Corollary \ref{eva1} to $p^\infty$-Selmer groups of elliptic curves  to deduce Corollary \ref{mainc} in subsection \ref{apeq}. 

\section {Proof of  Main Theorem}\label{2section}
 We start with the following remark
\begin{remark}
 By our assumption on $G$, $\Lambda_{O}(G)$ is (left and right) noetherian. it is then well known that the $U$-Euler characteristic of a finitely generated, $p$-primary  torsion $\Lambda_{O}(G)$  module always exists \cite[Proposition 1.6]{sh}. Thus to establish the Main Theorem, without any loss of generality, we may assume that $M$ is a finitely generated $\Lambda_{O}(H)$-module.

\end{remark}

We next prove the following lemma and use it to make Remark \ref{rednp}.
\begin{lemma}\label{prop}
Let $G$ be a compact $p$-adic Lie group without any element of order $p$ and let $W$ be any  normal subgroup of $G$. Let $M$ be a finitely generated $\Lambda_{O}(G)$-module. If $W$-Euler characteristic of $M$ exists, then $G$-Euler characteristic also exists. 
\end{lemma}
\begin{proof}
Since $W$-Euler characteristic of $M$ exists, by definition we have the groups $H_j(W,M)$ are finite for every integer $j\geq 0$. Thus $H_i \big(G/W,H_j(W,M)\big)$ is finite for every integer $i,j \geq 0$ (\cite{la}). Now it follows from the Hochschild-Serre spectral sequence that $H_r(G,M)$ is finite for every integer $r\geq 0$. Since $G$ is a compact $p$-adic  Lie group without any element of order $p$, the ring  $\Lambda_{O}(G)$ has finite global dimension. Consequently there exists an integer $m\geq 0$ such that $H_r(G,M)=0$ for all $r\geq m$. This shows that the $G$-Euler characteristic exists for $M$. 
\end{proof}

\begin{remark}\label{rednp}
By a result of Lazard \cite[Chapter 3, \S 3.1]{la}, any compact $p$-adic Lie group $G$ has an open normal uniform subgroup $W$. In particular, $W$ is a pro-$p$, $p$-adic Lie group without any element of order $p$. We choose and fix such a $W$. Now let $M$ be a  finitely generated $\Lambda_{O}(G)$ module. If  for every open normal subgroup $V$ of  $W$,  $V$-Euler characteristic of $M$ exists,  then it follows from Lemma \ref{prop} that for any open normal subgroup $U$ of $G$,  $U$-Euler characteristic of $M$ also exists. Hence from  this discussion, we further reduce Main Theorem to the special case where $G$ is a compact, pro-$p$  $p$-adic Lie  group without any element of order $p$. 

\end{remark}

\begin{lemma}\label{res}
Let $G$ be a pro-$p$, $p$-adic Lie group without any element of order $p$  and $M$ be a finitely generated $\Lambda_{O}(G)$-module which is also a finitely generated $\Lambda_{O}(H)$-module. Then there exist an open subgroup $G^{00}$ of $G$ with $ H \subset G^{00}$ and a resolution
\[  0 \lrta N_k \lrta N_{k-1} \lrta \cdots \lrta N_1 \lrta M \lrta 0 \]
of $M$ by finitely generated $\Lambda_{O}(G^{00})$-modules $N_i$, $i=1, \cdots,  k$, such that each $N_i$ is a free $\Lambda_{O}(H)$-modules.
\end{lemma}
\begin{proof}We make two observations which are used in the proof of this lemma. 

First, as  $H$  does not have any element of order $p$ either, the global dimension of $\Lambda_{O}(H)$ is also finite. Moreover, $H$ is also pro-$p$ group, $\Lambda_{O}(H)$ is a local ring and hence any finitely generated projective  $\Lambda_{O}(H)$ module is a finitely generated free $\Lambda_{O}(H)$ module. 

Second, we recall a basic fact from  homological algebra (see for example \cite[(5.6), Page 37]{ov}). Given an exact sequence of $\Lambda_{O}(H)$ modules, $$0 \lra K \lra M \lra L\lra 0,$$ let $k,m,l $ respective be the $\Lambda_{O}(H)$ projective dimensions of $K, M,L$.  we have $ m \leq \text{max} \{k, l \}$ and if this is an strict inequality, then  $k = l-1$.

Note if  the given  $M$ in the lemma is a projective $\Lambda_{O}(H)$-module then $N_1 =M$ works. Let $t  > 0$ be the $\Lambda_{O}(H)$-projective dimension of $M$ (which is finite by first observation).   By \cite[Key Lemma]{joz},
there exists an open subgroup $G^0_1$ of $G$, a finitely generated $\Lambda_{O}(G^0)$-module $N_1$ which is a free $\Lambda_{O}(H)$-module and an exact sequence of $\Lambda_{O}(G^0)$-modules 
\[  0 \lrta K_1 \lrta N_1 \lrta M \lrta 0. \]
If $M$ is not a projective  $\Lambda_{O}(H)$ module, then the $\Lambda_{O}(H)$-projective dimension of $K_1$ is equal to $d-1$, by our second observation. If $K_1$ is projective over $\Lambda_{O}(H)$, then we are done. Otherwise, again we apply \cite[Key Lemma]{joz} to get an open subgroup $G^0_2$ of $G^0_1$ containing $H$ and an exact sequence 
\[   0 \lrta K_2 \lrta N_2 \lrta K_1  \lrta 0 \] 
of $\Lambda_{O}(G^0_2)$-modules  such that $N_2$ is a projective $\Lambda_{O}(H)$-module. The $\Lambda_{O}(H)$-projective dimension of $K_2$ is equal to $d-2$. By repeating this process, at most $d$ times, we get an open subgroup $G_0^k$ of $G$ containing $H$ and an exact sequence
\[  0 \lrta N_k \lrta N_{k-1} \lrta \cdots \lrta N_1 \lrta M \lrta 0 \]
of  $\Lambda_{O}(G_0^k)$-module such that each $N_i$ is a finitely generated free  $\Lambda_{O}(H)$-module. This proves the lemma.

Note that in the special case of the false Tate curve extension (cf. \cite{za} for the definition of false Tate curve extension) where $G \cong \Z_p \rtimes \Z_p$, Lemma \ref{res} was proved in \cite[Lemma 4.3]{za}.
\end{proof}

Let $M$ be a  $\Lambda_{O}(G)$-module which is also a finitely generated free $\Lambda_{O}(H)$-module. Fix a $\Lambda_{O}(H)$ isomorphism $\phi : M \cong \Lambda_{O}(H)^d$. Let $\{e_1,e_2,\cdots,e_d\}$ be the standard basis of $ \Lambda_{O}(H)^d$, namely, $e_i$ denote the row vector with $1$ at $i$th entry and $0$ every where else. Since $M$ is a $G$-module, via the isomorphism $\phi$, $G$ acts on $ \Lambda_{O}(H)^d$. Let $A$ be the matrix of multiplication of $\tilde{\gamma}$ 
defined as $\tilde{\gamma}*e_i :=\Sigma_j a_{ij}e_j $ for $a_{ij}\in \Lambda_{O}(H)$. Thus we have, $\tilde{\gamma}*e_i = e_iA$. Let $x = \Sigma_k a_k e_k\in \Lambda_{O}(H)^d$. We get 

\begin{equation}\label{for}
 \tilde{\gamma}*x = \Sigma_k (\tilde{\gamma}a_k\tilde{\gamma}^{-1}) \tilde{\gamma} *e_k  = (\Sigma_k \tilde{\gamma}a_k\tilde{\gamma}^{-1} e_k)A.
 \end{equation}
Let us  denote the left $\Lambda_{O}(G)$-module $\frac{\Lambda_{O}(G)^d}{\Lambda_{O}(G)^d(\tilde{\gamma}I_d -A)}$  by $\tilde{M}$. Here $I_d$ denote the $d\times d$ identity matrix. 
Consider the natural map 
\[ \psi : \Lambda_{O}(H)^d \lrta \frac{\Lambda_{O}(G)^d}{\Lambda_{O}(G)^d(\tilde{\gamma}I_d -A)} \]
induced by the natural inclusion  map $\Lambda_{O}(H)\lrta \Lambda_{O}(G)$.  It can be easily shown that $\psi$ is an isomorphism of left $\Lambda_{O}(H)$-modules. We shall next show that it is an isomorphism of $\Lambda_{O}(G)$-modules. Recall $\Lambda_{O}(H)^d$ is a $\Lambda_{O}(G)$-module via the isomorphism $\phi$. To show $\psi$ is a $\Lambda(G)$-linear it is enough to show $\psi(\tilde{\gamma} *x )=\tilde{\gamma}\psi(x)$ for all $x$. Let $x = \Sigma_k a_k e_k\in \Lambda_{O}(H)^d$. From \eqref{for}, we have
\[  \psi(\tilde{\gamma}*x) = \psi( (\Sigma_k \tilde{\gamma}a_k\tilde{\gamma}^{-1} e_k)A) = (\Sigma_k \tilde{\gamma}a_k\tilde{\gamma}^{-1} e_k)A \in \tilde{M}. \]
Now the right multiplication by $A$ in $\tilde{M}$ is same as the right multiplication by $\tilde\gamma I_d$ in $\tilde{M}$. Thus we have $(\Sigma_k \tilde{\gamma}a_k\tilde{\gamma}^{-1} e_k)A=(\Sigma_k \tilde{\gamma}a_k\tilde{\gamma}^{-1} e_k)\tilde{\gamma}I_d \in \tilde{M}$.  Using the scalar multiplication by $\tilde{\gamma}$ in $\Lambda_O(G)^{d}$,  $\tilde\gamma e_i = e_i \tilde\gamma I_d$, we deduce that $\psi(\tilde{\gamma}*x)=(\Sigma_k \tilde{\gamma}a_k\tilde{\gamma}^{-1} e_k)\tilde{\gamma}I_d = \Sigma_k\tilde{\gamma} a_k  e_k =  \tilde{\gamma} (\Sigma_k a_k e_k)  \in \tilde{M}$. This proves that $\psi(\tilde{\gamma} *x )=\tilde{\gamma}\psi(x)$ in $\tilde{M}$.  Thus we have proved Lemma \ref{iso}  generalizing \cite[Lemma 1]{joz}:
\begin{lemma}\label{iso}
Let $M$ be a  $\Lambda_{O}(G)$-module which is  a finitely generated free $\Lambda_{O}(H)$-module. Then there is a $\Lambda_O(G)$ module isomorphism
 $$M \cong  \frac{\Lambda_{O}(G)^d}{\Lambda_{O}(G)^d(\tilde{\gamma}I_d -A)}. \quad \quad ~  \qed$$
\end{lemma}

\begin{proposition}\label{dim}
Let $G$ be a compact, pro-$p$ $p$-adic Lie group without any element of order $p$ and $M$ be a finitely generated $\Lambda_{O}(G)$-module which is also a finitely generated free  $\Lambda_{O}(H)$-module of rank $d$. Then  $M$  has  $\Lambda_{O}(G)$-projective dimension  equal to  $1$. 
\end{proposition}
\begin{proof}
Since  $M$ is a finitely generated $\Lambda_{O}(H)$-module, it is a torsion $\Lambda_{O}(G)$-module. In particular, it is not a free $\Lambda_{O}(G)$-module. Thus the $\Lambda_{O}(G)$ projective dimension of $M$  $\geq 1$. From Lemma \ref{iso}, there is a $\Lambda_{O}(G)$-isomorphism, $M\cong \frac{\Lambda(G)^d}{\Lambda(G)^d(\tilde{\gamma}I_d -A)}$ for a $d\times d$ matrix $A$ with entries in $\Lambda_{O}(H)$. Thus we have a resolution 
\begin{equation}\label{re}
  \Lambda_{O}(G)^d \stackrel{f}{\lrta}  \Lambda_{O}(G)^d \stackrel{q}{\lrta} M \lrta 0 
  \end{equation}
of left $ \Lambda_{O}(G)$-module where the second map $f$ is multiplication by $\tilde{\gamma}I_d -A$ from right.  Set $K = \mathrm{Ker}(f)$. We claim that $K$ is a (finitely generated) $ \Lambda_{O}(G)$ torsion module. Note that as $G$ is a compact $p$-adic Lie group, $ \Lambda_{O}(G)$ is a (left and right) noetherian ring  and as $G$ has no element of order $p$, $ \Lambda_{O}(G)$ has no non-zero left or right  zero divisors. It follows that the total skew field of fraction $ Q_{O}(G)$ of $ \Lambda_{O}(G)$ exists (cf. \cite[\S 1.1]{sh}). As $M$ is a torsion $ \Lambda_{O}(G)$ module, $ M \otimes_{\Lambda_{O}(G)}Q_{O}(G) =0$. Hence applying $-\otimes_{\Lambda_{O}(G)}Q_{O}(G)$  on \eqref{re}  and considering the additive property of the dimension of finite dimensional $Q_{O}(G)$ vector spaces in a short exact sequence, we deduce that $K\otimes_{\Lambda_{O}(G)}Q_{O}(G) =0$. This, in turn proves the claim that $K$ is  a $ \Lambda_{O}(G)$ torsion module. As $\Lambda_{O}(G)^d$ is a free $ \Lambda_{O}(G)$ module, it  has no non-zero $\Lambda_{O}(G)$ torsion submodule. Hence,  we deduce  that $K=0$. 
\begin{equation}\label{re3}
 0 \lrta  \Lambda_{O}(G)^d \stackrel{f}{\lrta}  \Lambda_{O}(G)^d \stackrel{q}{\lrta} M \lrta 0 
  \end{equation}
  is a short exact sequence and therefore the projective dimensions of $M(\rho)$ as a $ \Lambda_{O}(G)$-module is $\leq 1$.  This completes the proof of the proposition.
\end{proof}

\begin{proposition}\label{forfree}
Let $G$ be  a pro-$p$, $p$-adic Lie group without any element of order $p$ and $M$ be a finitely generated $\Lambda_{O}(G)$-module which is also a finitely generated free  $\Lambda_{O}(H)$-module of rank $d$. Then  there is a countable subset $S_M$ of the set $S : = \{ \text{all continuous characters from  }\Gamma \lra \ZZ_p^\times\}$ such that if we choose and fix any $\rho \in S \setminus S_M$,  the $U$-Euler characteristic of $M(\rho)$ exists for every open normal subgroup $U$ of $G$ and we have $\chi(U,M(\rho))=\#H_0(U,M(\rho))$. 
\end{proposition}
\begin{proof}
From Lemma \ref{dim}, we have a $\Lambda_{O}(G)$ projective resolution of $M$ of the form
\[  0 \lrta \Lambda_{O}(G)^d \lrta   \Lambda_{O}(G)^d \lrta M \lrta 0 . \]
Twisting by a character $\rho$ of $\Gamma$ and taking co-invariance by an open normal subgroup $U$, we get and exact sequence
\[  0 \lrta H_1(U,M(\rho)) \lrta O^{d[G:U]} \lrta  O^{d[G:U]} \lrta H_0(U,M(\rho)) \lrta 0 . \]
 From Theorem \ref{known}, we get that $H_0(U,M(\rho))$ is finite for every $U$ once we have chosen  $\rho$ as in the statement of this proposition. This implies that $H_1(U,M(\rho))=0$. This proves the proposition.
\end{proof}
We are now ready to prove the Main Theorem.

{\it Proof of Main Theorem:}
By Lemma \ref{res}, there exists an open normal subgroup $G^{00}$ of $G$ containing $H$, a resolution
\begin{equation}\label{reso}
 0 \lrta N_k \lrta N_{k-1} \lrta \cdots \lrta N_1 \lrta M \lrta 0 
\end{equation}
 of $M$ by finitely generated $\Lambda_{O}(G^{00})$-modules $N_i$, $i=1, \cdots, k$ such that each $N_i$ is a free $\Lambda_{O}(H)$-modules.
Let $\Gamma^{00}:= G^{00}/H$. 

Since $\Gamma^{00}$ is a finite $p$-power index subgroup of $\Gamma$ and $\ZZ_p^\times$ has no element of order $p$, the following natural restriction map in \eqref{rest} is injective.
\begin{equation}\label{rest}
 \text{res} :~  \Hom(\Gamma,\ZZ_p^\times) \lrta \Hom(\Gamma^{00},\ZZ_p^\times) 
\end{equation}

For each $N_i$ there is a countable subset $S_i \subset \Hom(\Gamma^{00},\ZZ_p^\times)$ such that if we choose and fix any $\rho \in \Hom(\Gamma^{00},\ZZ_p^\times) \setminus  S_i$, then  $\chi(U,N_i(\rho))$ exists for every open normal subgroup $U$ of $G^{00}$.  Hence by choosing a    $\rho \in  \Hom(\Gamma^{00},\ZZ_p^\times) \setminus  \underset{1 \leq i \leq n} {\cup}S_i$,   simultaneously for every $i$, $\chi(U,N_i(\rho))$ exists for all open normal subgroup $U$ of $G^{00}.$ This in turn  from \eqref{reso}, implies that for any such chosen $\rho: \Gamma^{00} \lra \ZZ_p^\times$, $\chi(U, M(\rho))$ exists for each open normal subgroup $U$ of $G^{00}$.

Let $U$ be an open normal subgroup of $G$ and choose and fix a $\rho: \Gamma \lra \ZZ_p^\times$ such that $\rho|_{\Gamma^{00}} \notin   \underset{1 \leq i \leq n} {\cup}S_i$. Then by preceding discussion, $\chi\big(U\cap G^{00},M(\rho|_{\Gamma^{00}})\big)$-exists. Moreover, applying Lemma \ref{prop} and \eqref{rest}, it follows that $\chi(U,M(\rho))$ also exist. This completes the proof. \qed

 \section{Application}\label{3section}
 \subsection{Akashi Series and Euler Characteristic}\label{appli1}
Let $G$ be as before and set $ \Lambda(G) : =\Lambda_{\mathbb{Z}_p}(G)$. Moreover for $G =\Gamma$, we will sometime write $\Lambda :=  \Lambda(\Gamma)$. Fix an isomorphism $\Lambda \cong \ZZ_p[[X]]$ by sending $\gamma$ to $1+X$.  For a  finitely generated torsion $\Lambda(\Gamma)$-module $N$, let $char_{\Lambda(\Gamma)}(N)$  denote its usual characteristic element  $ \in (\Lambda(\Gamma) \setminus \{0\})/(\Lambda(\Gamma))^\times$. Write $Q(\Lambda(\Gamma))$ for the quotient field of $\Lambda(\Gamma)$. We recall the following definition of Akashi Series.

\begin{definition}[cf. \cite{CFKSV}]
Let $M \in \mathfrak M_H(G)$. We say the {\it Akashi Series  } of $M$ with respect to $H$ exists, if the Galois homology groups $H_i(H,M)$ are finitely generated torsion $\Lambda(\Gamma)$-modules for every $i\geq 0$.   If the Akashi series of $M$ exists, then it is given by 
\begin{equation}
Ak^G_H(M) = \prod_i \big(char_{\Lambda(\Gamma)} H_i(H,M)\big)^{(-1)^i} \in Q(\Lambda(\Gamma))^\times/\Lambda(\Gamma)^\times
\end{equation}
\end{definition}

\begin{lemma}\label{exist}
Let $M\in \mathfrak M_H(G)$. Then $Ak_{H\cap U}^U(M)$ exists for every open normal subgroup $U$ of $G$. 
\end{lemma}
\begin{proof}
Put $T = H \cap U$. Then $T$ is a finite index subgroup of $H$ and therefore  $\Lambda(H)$ is a finite module over $\Lambda(T)$. In particular, $M/M(p)$ is a finitely generated $\Lambda(T)$-module. Therefore from \cite[Lemma 3.1]{CFKSV}, $Ak_T^U(M)$ exist. 
\end{proof}

For  a character  $\rho : \Gamma \lrta \ZZ_p^\times $,  let  $\tilde{\rho} : \Lambda(\Gamma) \lrta \ZZ_p$ denote the  corresponding ring homomorphism induced by the evaluation map.   We extension this homomorphism to 
\[   \tilde{\rho} :  Q(\Lambda(\Gamma))^\times \lrta \mathbb Q_p \cup \{\infty\} \]
as $\tilde{\rho}(f/g)=\tilde{\rho}(f)/\tilde{\rho}(g)$ if $f$ and $g$ have no common divisor and $\tilde{\rho}(g)\neq 0$. If $f$ and $g$ have no common divisor and $\tilde{\rho}(g)=0$ then put $\tilde{\rho}(f/g)=\infty$. For simplicity of the notation we denote $\widetilde{\rho^{-1} }$ by $\tilde{\rho}^{-1}$.

\begin{definition}
We say $\tilde{\rho}(f/g)$ exists if $\tilde{\rho}(f/g)\neq 0,\infty$. 
\end{definition}

 Let $M$ be a finitely generated torsion $\Lambda(\Gamma)$-module. It can be shown that if $\chi(\Gamma,M)$ exists then $\chi(\Gamma,M)=char_{\Lambda(\Gamma)}(M)|_{X=0}=\tilde{\rho}(char_{\Lambda(\Gamma)}(M))$, where $\rho$ denotes the trivial character of $\Gamma$ (see \cite[Lemma 4.2]{Gr}). More generally we have the following lemma

 \begin{lemma}\label{twist}
If $\rho$ is an arbitrary character of $\Gamma$ and $\chi(\Gamma,M(\rho))$ exists then $\chi(\Gamma,M(\rho))=\tilde{\rho}^{-1}(char_{\Lambda(\Gamma)}(M))$.
 \end{lemma}
\begin{proof}
For a character $\tau : \Gamma \lrta \ZZ_p^\times$ define $Tw_\tau : \Lambda \lrta \Lambda $ be the $\ZZ_p$-linear isomorphism induced by 
\[   Tw_\tau(\gamma) = \tau(\gamma)\gamma  \] 
For a finitely generated torsion $\Lambda$-module $B$ from \cite[Chapter VI, Lemma 1.2]{R} we have
\[  Tw_\tau(char_\Lambda(B\otimes \tau))  = char_\Lambda(B) \]
Taking $B=M\otimes \rho$ and $\tau=\rho^{-1}$ we get
\begin{equation}\label{tw}
  Tw_{\rho^{-1}}(char_\Lambda(M))=char_\Lambda(M\otimes \rho) 
 \end{equation} 
We get
\begin{equation}\label{tw1}
 \tilde{\rho}^{-1}(char_\Lambda(M))= Tw_{\rho^{-1}}(char_\Lambda(M))|_{X=0}=char_\Lambda(M\otimes \rho)|_{X=0}=\chi(\Gamma,M(\rho)).
\end{equation}
\end{proof}
 
\begin{lemma}\label{eva}
Let $M  \in \mathfrak M_H(G)$ and $U$ be an open normal subgroup of $G$. For a character $\rho : \Gamma \lrta \ZZ_p^\times $ if $\chi(U,M(\rho))$ exist then  $ \tilde{\rho}^{-1}(Ak_{U\cap H}^U(M)) $ exits, too, and  $\chi(U,M(\rho)) = \tilde{\rho}^{-1}(Ak_{U\cap H}^U(M))$.
\end{lemma}
\begin{proof}
As $M/M(p)$ is a finitely generated $\Lambda(H)$-module;  $\frac{M(\rho)}{(M(\rho))(p)}$ is also a finitely generated $\Lambda(H)$-module. 
Therefore  $Ak_{U\cap H}^U(M(\rho))$ exists. Since $\chi(U,M(\rho))$ exists it follows  from \cite[Lemma 4.2]{CSS} that $Ak_{U\cap H}^U(M\otimes \rho))|_{X=0} $  is finite,  non-zero and  
\begin{equation}\label{tw2}
\chi(U,M(\rho))=  Ak_{U\cap H}^U(M(\rho))|_{X=0}  .
\end{equation}  
We have for each $i\geq 0$, $H_i(U\cap H,M(\rho))\cong  H_i(U\cap H,M)(\rho)$. Therefore from \eqref{tw} we get that 
\begin{equation}\label{tw3}
Ak_{U\cap H}^U(M(\rho))=Tw_{\rho^{-1}}(Ak_{U\cap H}^U(M)).
\end{equation}
Since $Ak_{U\cap H}^U(M\otimes \rho))|_{X=0} $ is finite and non-zero, $Tw_{\rho^{-1}}(Ak_{U\cap H}^U(M))|_{X=0} = \tilde{\rho}^{-1}(Ak_{U\cap H}^U(M))$ also exists and from \eqref{tw2} and \eqref{tw3} we get that 
  $\chi(U,M(\rho)) = \tilde{\rho}^{-1}(Ak_{U\cap H}^U(M))$.  This proves the lemma.
\end{proof}
As a consequence of Main Theorem and Lemma \ref{eva}, we have
\begin{corollary}\label{eva1}
Let $M \in \mathfrak M_H(G)$. Then there is a countable subset $S_M$ of the set all continuous characters from  $\Gamma \lra \ZZ_p^\times$ such that if we choose and fix any character $\rho $ outside $ S_M$, then   $\tilde{\rho}^{-1}(Ak_{U\cap H}^U(M))$ exists for all open normal subgroups $U$ of $G$ and satisfies $\chi(U,M(\rho)) = \tilde{\rho}^{-1}(Ak_{U\cap H}^U(M))$. \qed
\end{corollary}

 \subsection{Relation to the Iwasawa theory of elliptic curves}\label{apeq}

Throughout this section $E$ will be  an elliptic curve defined over a number field $K$ with good and ordinary reduction at all primes of $K$ dividing $p$. The $p^\infty$-Selmer group $S_p(E/L)$ of $E$ over a finite extension $L/K$ is defined as
\[ 0 \lrta S_p(E/L) \lrta H^1(L,E(p)) \lrta \prod_v H^1(L_v,E)  \]
where $E(p): = \underset{n \geq 1}{\cup}E(\bar{\Q})[p^n]$ and $v$ varies over set of finite primes of $L$. If $L_\infty/K$ is an infinite algebraic extension of $K$ then we define 
\[  S_p(E/L_\infty) :=  \varinjlim_L S_p(E/L) \]
where $L$ varies over number fields $L\subset L_\infty$ and the direct limit is taken over the natural corestriction maps.  
The natural action of $\gal(L_\infty/K)$  on $S_p(E/L_\infty)$ (endowed with discrete topology) makes  it a discrete $\Lambda(\mathrm{Gal}(L_\infty/K))$ module. For a $p$-adic Lie extension $L_\infty$, the  Pontryagin dual $$S_p(E/L_\infty)^\vee:= \mathrm{Hom}_{\text{cont}}\big (S_p(E/L_\infty), \Q_p/\Z_p\big )$$ is a finitely generated (see for example \cite{ov}) $\Lambda(\mathrm{Gal}(L_\infty/K))$ module.
 \begin{definition} 
 \begin{enumerate}
 \item A Galois extension $\K$  of $K$ is called {\it admissible} if  (i) $\K$ is unramified outside finitely many primes of $K$, (ii) $G:= \gal(\K/K)$ is a $p$-adic Lie group without any element of order $p$,  (iii) $K_{cyc}:= K \Q_\cyc \subset \K$. We write $\Gamma := \gal(K_{cyc}/K)$, $H := \gal(\K/K_{cyc})$.  
 \item An admissible extension $\K/K$ is said to be strongly admissible if for every prime $v|p$ of $K$ and a prime $w|v$ of $\K$, the completion at $w$, $K_{\infty, w}$ contains the unramified $\ZZ_p$-extension of $K_v$.  
\end{enumerate}
\end{definition}

\begin{conjecture}[$\mathfrak M_H(G)$ Conjecture; Conjecture 5.1, \cite{CFKSV}]
Let $E$ and $p$ be as above. Let $K_\infty/K$ be admissible. Then the $\Lambda(G)$-module  $S_p(E/\K)^\vee$ lies in $ \mathfrak M_H(G)$.  (Note that for $K_\infty = K_\cyc$ this is precisely Mazur's conjecture.)
\end{conjecture}

If $S_p(E/K_{cyc})^\vee$ is a finitely generated $\ZZ_p$-module then it can be shown that $S_p(E/K_\infty)^\vee\in \mathfrak M_H(G)$  (\cite[Proposition 5.6]{CFKSV}). In particular, if  $S_p(E/K_{cyc})^\vee$ is a finitely generated $\ZZ_p$-module then  $Ak_{U\cap H}^U(S_p(E/K_\infty)^\vee)$ exists for every open normal subgroup $U$.

For an open normal subgroup $U$ of $G=\gal(\K/K)$,  $K_U:= K_\infty^U$ is the fixed field of $\K$ by $U$. We write $K_{U,\cyc}: = K_UK_\cyc$. 
Note that $H\cap U = \gal(\K/K_{U,\cyc})$. Put $\Gamma_U= \gal(K_{U,\cyc}/K_U) \cong \Z_p$. For a number field $L$ and prime $v$ of $L$, set
\[  J_v(E/L_{cyc}):=  \prod_{w|v} H^1(L_{cyc,w},E(p)) \]
where $w$ varies over the set of primes of $L_{cyc}$ lying over $v$.  The Pontryagin dual of $J_v(E/L_{cyc})$, denoted by $X_v(E/L_{cyc})$ is  a finitely generated $\ZZ_p$-module and hence a torsion $\Lambda(\gal(L_{cyc}/L))$-module. 

Let $E$ and $p$ be as before and take  $\K/K$ to be a strongly admissible extension. Assume  $\frac{S_p(E/K_\infty)^\vee}{S_p(E/K_\infty)^\vee(p)}$ is finitely generated over $\Lambda(H)$. As a consequence of  \cite[Theorem 1.3]{Z}, we have  for every open normal subgroup $U$ of $G$,

\begin{equation}\label{akashi}
Ak_{U\cap H}^U(S_p(E/K_\infty)^\vee) = char_{\Lambda(\Gamma_U)}(S_p(E/K_{U,cyc})^\vee) \prod_{v\in \Sigma_U} char_{\Lambda(\Gamma_U)}(X_v(E/K_{U,\cyc}))
\end{equation}

with $\Sigma_U = \{v \text{ prime in } K_U \mid  I_v(\K/K_U) \text{ is infinite}\}$ and $I_v(\K/K_U)$ denote the inertia subgroup of $\gal(\K/K_U)$ at $v$. 

Using Corollary \ref{eva1} and Lemma \ref{twist} in \eqref{akashi}, we deduce an application of Main Theorem.
\begin{corollary}\label{mainc}
Let $E$ and $p$ be as before and take  $\K/K$ to be a strongly admissible extension with $G =\gal(\K/K)$. Assume  $\frac{S_p(E/K_\infty)^\vee}{S_p(E/K_\infty)^\vee(p)}$ is finitely generated over $\Lambda(H)$. Then there is a countable subset $S_M$ of the set all continuous characters from  $\Gamma \lra \ZZ_p^\times$, such that if we choose and fix any character $\rho $ outside $ S_M$, then for every open normal subgroup $U$ of $G$,  $\chi \big(U,   S_p(E/  \K)^\vee(\rho)\big) $ exists and 
\begin{align}
& \chi \big(U,    S_p(E/  \K)^\vee(\rho)\big) \nonumber \\ 
& =  \tilde{\rho}^{-1} \big(Ak_{U\cap H}^U(S_p(E/K_\infty)^\vee)\big)   \nonumber \\
  & =  \tilde{\rho}^{-1}\big(char_{\Lambda(\Gamma_U)}(S_p(E/K_{U,cyc})^\vee)\big) \prod_{v\in \Sigma_U} \tilde{\rho}^{-1}\big(char_{\Lambda(\Gamma_U)}(X_v(E/K_{U,\cyc}))\big) \nonumber
\end{align}
with $\Sigma_U = \{v \text{ prime in } K_U \mid  I_v(\K/K_U) \text{ is infinite}\}$. \qed
\end{corollary}
We refer to the discussions in \cite[\S 1.2]{Z} where the strong admissibility condition in Corollary \ref{mainc} can be replaced by a different set of hypotheses. 

\section*{Acknowledgement} We thank the referee  for his/her  comments, suggestions and in particular for simplifying the proof of Proposition \ref{dim}. The first named author gratefully acknowledges the support of DST INSPIRE faculty award grant and SERB ECR grant and the second named author gratefully acknowledges the support of DST INSPIRE faculty award grant.

\end{document}